\documentclass[10pt]{amsart}

\usepackage{amsmath,amssymb,fullpage,enumitem}
\usepackage{color}
\usepackage{ifthen}
\usepackage{theoremref}

\newboolean{usecolor}
\setboolean{usecolor}{true}

\newcommand\startcolor{\setboolean{usecolor}{true}}

\ifx\color@rgb\@undefined\else
        \definecolor{violet}{rgb}{0.25,0,0.75}
        \definecolor{green}{rgb}{0,0.7,0}
\fi

\newcommand\blue[1]{{\ifusecolor\color{black}\fi{#1}}}

\newcommand\magenta[1]{{\ifusecolor\color{black}\fi{#1}}}
\newcommand\violet[1]{{\ifusecolor\color{black}\fi{#1}}}

\newcommand\defi[1]{\blue{\sl #1}}

\newcommand\needref[1]{\magenta{???}}
\newcommand\needcite[1]{\magenta{[???]}}

\newcommand\C{\violet{C}}
\newcommand\U{\violet{U}}
\newcommand\V{\violet{V}}
\newcommand\X{\violet{X}}
\newcommand\Y{\violet{Y}}

\newcommand\Up{\blue{U'}}
\newcommand\Pone{\violet{\bbP^1}}
\newcommand\Aone{\violet{\bbA^1}}

\newcommand\Z{\violet{Z}}

\newcommand\ubar{\violet{\bar{u}}}

\newcommand\zbar{\violet{\bar{z}}}
\newcommand\etabar{\violet{\bar\eta}}

\newcommand\bbA{{\mathbb{A}}}
\newcommand\bbC{{\mathbb{C}}}
\newcommand\bbF{{\mathbb{F}}}
\newcommand\bbG{{\mathbb{G}}}

\newcommand\bbP{{\mathbb{P}}}
\newcommand\bbQ{{\mathbb{Q}}}
\newcommand\bbR{{\mathbb{R}}}
\newcommand\bbZ{{\mathbb{Z}}}

\renewcommand{\L}{{\Lambda}}

\newcommand\m{\violet{\mu}}

\renewcommand{\b}{\violet{\beta}}
\renewcommand{\d}{\violet{\partial}}

\renewcommand{\l}{\lambda}
\renewcommand{\t}{\violet{\tau}}

\newcommand\FF{\violet{\mathcal{F}}}
\newcommand\GG{\violet{\mathcal{G}}}
\newcommand\HH{\violet{\mathcal{H}}}
\newcommand\II{\violet{\mathcal{I}}}

\newcommand\LL{\violet{\mathcal{L}}}
\newcommand\TT{\violet{\mathcal{T}}}

\newcommand\act{\violet{\mid}}

\newcommand\Gm{\violet{\bbG_m}}

\newcommand\Fp{\violet{\bbF_p}}
\newcommand\Fq{\violet{\bbF_q}}
\newcommand\Fqn{\violet{\bbF_{q^n}}}
\newcommand\Fqbar{\violet{\bar\bbF_q}}

\newcommand\E{\violet{E}}
\newcommand\Ebar{\violet{\bar{E}}}
\newcommand\Et{\violet{E^\times}}
\newcommand\Ep{\violet{E'}}

\newcommand\Znn{\violet{\bbZ_{\geq0}}}
\newcommand\Zpos{\violet{\bbZ_{>0}}}

\newcommand\piOne[1]{\violet{\pi_1(#1)}}

\newcommand\piOneU{\violet{\piOne{\U}}}
\newcommand\piOneV{\violet{\piOne{\V}}}

\newcommand\End{\violet{\operatorname{End}}}
\newcommand\Fr{\violet{\operatorname{Fr}}}

\newcommand\Hom{\violet{\operatorname{Hom}}}
\newcommand\Ind{\violet{\operatorname{Ind}}}

\newcommand\ps\s
\newcommand\rank{\violet{\operatorname{rank}}}
\newcommand\s{\violet{\operatorname{Swan}}}
\newcommand\swp{\violet{\operatorname{SwanPolygon}}}
\newcommand\swv{\violet{\operatorname{SwanVertices}}}

\newcommand\tr{\violet{\operatorname{tr}}}
\newcommand\ub{\violet{\operatorname{unipotentBlocks}}}
\newcommand\um{\violet{\operatorname{unipotentMultiplicity}}}
\newcommand\unip{\violet{\operatorname{Unip}}}
\newcommand\w{\violet{\operatorname{weights}}}
\newcommand\wm{\violet{\operatorname{weightMultiplicity}}}

\renewcommand{\c}{\violet{\operatorname{cond}}}
\renewcommand{\d}{\violet{\operatorname{drop}}}

\renewcommand{\r}{\violet{\operatorname{rank}}}

\newcommand\FFL{\violet{\FF_\L}}

\newcommand\FFl{\violet{\FF_\l}}
\newcommand\FFlG{\violet{\FF^G_\l}}
\newcommand\FFlu{\violet{\FF_{\l,\etabar}}}

\newcommand\FFlz{\violet{\FF_{\l,\zbar}}}

\newcommand\GGl{\violet{\GG_\l}}
\newcommand\GGlu{\violet{\GG_{\l,\etabar}}}
\newcommand\GGL{\violet{\GG_\L}}

\newcommand\Vl{\violet{V_\l}}
\newcommand\VlIz{\violet{V_\l^{I(z)}}}
\newcommand\VlIc{\violet{V_\l^{I(c)}}}
\newcommand\VlIx{\violet{V_\l^{I(x)}}}

\newcommand\VlG{\violet{V_\l^G}}

\newcommand\DM{\violet{M^\vee}}
\newcommand\Ml{\violet{M_\l}}
\newcommand\DMl{\violet{M^\vee_\l}}

\newcommand\EGL{\violet{(E[G],\L)}}
\newcommand\El{\violet{E_\l}}
\newcommand\Elbar{\violet{\bar{E}_\l}}
\newcommand\EL{\violet{(E,\L)}}
\newcommand\EG{\violet{E[G]}}
\newcommand\ElG{\violet{\El[G]}}

\newcommand\ElIz{\violet{\El[I(z)]}}

\newcommand\smallsum{\blue{\mbox{$\Sigma$}}}

\newcommand\smallfrac[2]{\blue{\mbox{$\frac{#1}{#2}$}}}

\newcommand\ssm\smallsetminus
\newcommand\sub\subset
\newcommand\seq\subseteq
\newcommand\sneq\subsetneq

\newcommand\etale{\violet{\'etale}}
\newcommand\indep{\violet{independent of $\l$}}

\newcommand\np{\newpage}

\newtheorem{thm}{Theorem}
\newtheorem{lem}[thm]{Lemma}
\newtheorem{lemma}[thm]{Lemma}
\newtheorem{cor}[thm]{Corollary}
\newtheorem{prop}[thm]{Proposition}

\newcommand\myitem{\rm\item\it}
\newenvironment{enum}%
	{%
	 \begin{enumerate}\setlength{\itemsep}{0.075in}}%
	{\end{enumerate}%
	}

\begin{document}


\!\!\!\!
\begin{abstract}
We show that certain ramification invariants associated to a compatible system of $\ell$-adic sheaves on a curve are independent of $\ell$.
\end{abstract}


\startcolor

\numberwithin{thm}{section}

\title{Ramification of compatible systems on curves and independence of $\ell$}
\author{Chris Hall}
\email{chall69@uwo.ca}
\date\today
\thanks{The research for this paper was carried out while the author was a von Neumann fellow at the IAS}
\maketitle



\section{Introduction}

Let $p$ be a prime, $\Fq/\Fp$ be a finite extension, and $\C/\Fq$ be a proper smooth geometrically connected curve.  Let $\Z\subset\C$ be a finite subset, $\U=\C\ssm\Z$ be the open complement, and $j\colon\U\to\C$ be the natural inclusion.  Let $\etabar$ be a geometric generic point of $U$ and $\piOneU=\pi_1(\U,\etabar)$ be the etale fundamental group.  For each closed point $c\in\C$, let $I(c)\seq D(c)\seq\piOneU$ be an inertia and decomposition subgroups, $P(c)\seq I(c)$ be the $p$-Sylow subgroup, and $\Fr_c\in D(c)$ be an element mapping to the geometric Frobenius element in $D(c)/I(c)$.

Let $\E/\bbQ$ be a number field, $\bbZ_\E$ be its ring of integers, and $\L$ be a set of non-zero primes $\l\sub\Z_\E$ not dividing $p$.  For each $\l\in\L$, let $\FFl$ be a lisse sheaf on $\U$ of $\El$-modules, let $\Vl$ be the geometric generic fiber $\FFlu$, and let $\FFL=\{\FFl\}_{\l\in\L}$ be the corresponding family of sheaves.  For each closed point $c\in\C$, let
$$
	L(T,\FF_{\l,c})
	=
	\det(1 - T\,\Fr_u\act\VlIc).
$$
We say that $\FFL$ is \defi{$\EL$-compatible} iff for every closed point $u\in\U$, the coefficients of $L(T,\FF_{\l,c})$ all lie in $\E$ and are \indep.

For each $z\in\Z$, let $\swp_z(\FFl)$, $\swv_z(\FFl)$, and $\s_z(\FFl)$ be the Swan polygon, its set of vertices, and the Swan conductor respectively of $\Vl$ as an $\ElIz$-module.  We call these the \defi{Swan invariants} of $\FFl$ about $z$ and recall their definitions in \S\ref{sec:swan}.  Our first theorem is the following:

\begin{thm}\thlabel{thm:swan:semel}
Let $\Fq$ be a finite field, $\C/\Fq$ be a proper smooth geometrically connected curve, and $\U\seq\C$ be a dense Zariski open subset over $\Fq$.  Let $\E/\bbQ$ be a number field,  $\L$ be a set of non-zero primes $\l\sub\bbZ_\E$ not dividing $q$, and $\FFL=\{\FFl\}_{\l\in\L}$ be an $\EL$-compatible system of lisse sheaves on $\U$.  Then the following hold, for each $z\in\C\ssm\U$:

\smallskip
\begin{enum}
\myitem\label{it:thm:swan:cond} $\s_z(\FFl)$ is \indep;
\myitem\label{it:thm:swan:poly} $\swv_z(\FFl)$ is \indep, and thus so is $\swp_z(\FFl)$.
\end{enum}
\end{thm}

\noindent
Of course, independence of $\l$ is an assertion about each pair of primes $\l,\l'\in\L$, so there is no loss of generality in supposing $\L$ is finite.  Moreover, given $\swv_z(\FFl)$, one can easily determine $\s_z(\FFl)$, hence \thref{thm:swan:semel} is equivalent to the following theorem when $\C=\Pone$:

\begin{thm}\thlabel{thm:swan:bis}
Suppose that the hypotheses of \thref{thm:swan:semel} hold, that $\L$ is finite, and that $\C=\Pone$.  Then $\swv_z(\FFl)$ is \indep, for each $z\in\C\ssm\U$.
\end{thm}

\noindent
An initial step in our proof of \thref{thm:swan:semel} is to show that one can reduce to the case $\C=\Pone$, hence these two theorems are equivalent.  While knowing $\s_z(\FFl)$ is usually not enough to determine $\swv_z(\FFl)$, we show that nonetheless the following theorem implies both of the previous theorems:

\begin{thm}\thlabel{thm:swan:ter}
Suppose that the hypotheses of \thref{thm:swan:bis} hold.  Then $\s_z(\FFl)$ is \indep, for each $z\in\Pone\ssm\U$.
\end{thm}

\noindent
For the proofs of these theorems, see \S\ref{sec:proof-of-swan-thms}.

Given an integer $w$, we say that $\FFl$ is \defi{pointwise pure of weight $w$} iff for every closed point $u\in\U$, each zero $\alpha\in\Elbar$ of $L(T^{\deg(u)},\FF_{\l,u})$ lies in $\Ebar\sub\Elbar$ and satisfies
$
	|\iota(\alpha)|^2 = (1/q)^w
$
for every field embedding $\iota\colon\Ebar\to\bbC$.  We say that an $\EL$-compatible system $\FFL$ is \defi{pointwise pure of weight $w$} iff some (hence every) $\FFl$ is pointwise pure of weight $w$.

Let $z\in Z$ be a point and $\zbar\to z$ be a geometric point.  Let $\FFlz$ be the $\El$-module $(j_*\FFl)_\zbar$ and $\r_z(\FFl)$ be its $\El$-dimension.  For each positive integer $e$, let $\um_{e,z}(\Vl)$ be the largest non-negative integer $m$ such that some $\ElIz$-submodule of $\Vl$ is isomorphic to $\unip(e)^{\oplus m}$, and let
$$
	\ub_z(\FFl)
	=
	\{\,
		(e,m)
		:
		m=\um_{e,z}(\Vl)
		\mbox{ and }
		m>0
	\,\}
$$
be the set describing the structure of the maximal $\ElIz$-submodule of $\Vl$ where $I(z)$ acts unipotently.  Finally, let
\[
	\d_z(\FFl) := \r_\El(\FFl) - \r_z(\FFl),
	\quad
	\c_z(\FFl) := \d_z(\FFl)+\s_z(\FFl),
\]
and observe that if $\s_z(\FFl)$ and $\r_z(\FFl)$ are \indep, then so are $\d_z(\FFl)$ and $\c_z(\FFl)$.  We call these (and the Swan invariants) the \defi{ramification invariants} of $\FFl$ about $z$.

\begin{thm}\thlabel{thm:pure}
Suppose that the hypotheses of \thref{thm:swan:semel} hold and that $\FFL$ is pointwise pure of weight $w$.  Then, for each $z\in\C\ssm\U$, the following hold:

\smallskip
\begin{enum}
\setcounter{enumi}{1}
\myitem $\r_z(\FFl)$ is \indep{}, and thus so are $\d_z(\FFl)$ and $\c_z(\FFl)$;
\myitem $\ub_z(\FFl)$ is \indep{}.
\end{enum}
\end{thm}

\noindent
See \S\ref{sec:proof-of-purity-theorem} for a proof.

As a corollary of \thref{thm:pure} we obtain the following result which is also an immediate corollary of \cite[th.~9.8]{Deligne:Constantes} and \cite[Appendix]{Katz:WR}:

\begin{cor}
Under the hypotheses of \thref{thm:pure}, the truth of each of the following assertions is \indep:

\smallskip
\begin{enum}
\rm\item\label{li:tame}\it $\FFl$ has local tame monodromy about $z$;
\rm\item\label{li:unip}\it $\FFl$ has local unipotent monodromy about $z$.
\rm\item\label{li:triv}\it $\FFl$ has local trivial monodromy about $z$.
\end{enum}
\end{cor}

\noindent
Indeed, \eqref{li:tame} (resp.~\eqref{li:triv}) holds if and only if $\s_z(\FFl)=0$ (resp.~$\d_z(\FFl)=0$).  Moreover, \eqref{li:unip} holds if and only if
$
	\sum_{(e,m)\in\ub_z(\FFl)} em = \dim(V_\l).
$

Suppose now that each $\FFl$ is a lisse sheaf of $\El[G]$-modules on $\U$ for some common finite group $G$.  We define the notion of an $\EGL$-compatible system $\FFL$ and prove the following theorem in \S\ref{sec:EGL-compatible}:

\begin{thm}
Let $\Fq$ be a finite field, $\C/\Fq$ be a proper smooth geometrically connected curve, and $j\colon\U\to\C$ be the inclusion of a dense Zariski open subset over $\Fq$.  Let $\E/\bbQ$ be a number field, $G$ be a finite group, $\L$ be a set of non-zero primes $\l\sub\bbZ_\E$ not dividing $q$, and $\FFL=\{\FFl\}_{\l\in\L}$ be a system of lisse sheaves on $\U$.  If $\FFL$ is $\EGL$-compatible and pure of weight $w$, then $j_*\FFL$ is $\EGL$-compatible.
\end{thm}

\noindent
One can regard this as an equivariant version of Theorem~\ref{thm:pure}, and we prove it by reducing to the latter.


\subsection{Acknowledgements}

We gratefully acknowledge Nick Katz for explaining how to show (\`a la Deligne) that $\s_z(\FFl)$ and $\c_z(\FFl)$ are \indep{}, for suggesting we investigate whether or not Swan polygons are \indep, and for comments on early drafts.


\np
\section{Swan conductors and polygons}\label{sec:swan}

Suppose the hypotheses of Theorem~\ref{thm:swan:semel} hold, and fix $z$ in $\Z=\C\ssm\U$.  Recall $I(z)\seq\piOne{\U}$ is an inertia group and is defined up to conjugation.


\subsection{Single modules}

Let $I(z)^{(r)}$ be the subgroup indexed by the $r\geq 0$ in the upper-numbering filtration on $I(z)$ (cf.~\cite[1.0]{Katz:GKM}) and $P(z)\seq I(z)$ be the $p$-Sylow group.  Then $I(z)=I(z)^{(0)}$ and $I(z)\supset P(z)\supset I(z)^{(r)}\supset I(z)^{(s)}$ for $0<r<s$.

Let $M_\l$ be a non-zero finite-dimensional $\ElIz$-module.  As shown in \cite[1.1]{Katz:GKM}, there is a unique decomposition $M_\l=\oplus_{x\geq 0}M_\l(x)$ where the submodules $M_\l(x)\seq M_\l$ satisfy
\begin{equation}\label{eqn:break-submodules}
	M_\l(0)=M_\l^{P(z)},\quad
	(M_\l(x))^{I(z)^{(r)}}
	=
	\begin{cases}
		0    & x\geq r \\
		M_\l(x) & r>x.
	\end{cases}
\end{equation}
This decomposition is the \defi{break decomposition} of $M_\l$, and the \defi{breaks} $0\leq\b_1<\cdots<\b_m$ of $M_\l$ (also called \defi{slopes}) are defined to be the $x\geq 0$ satisfying $M_\l(x)\neq 0$ (cf.~\cite[1.2]{Katz:GKM}).

\begin{lem}\thlabel{lem:tensor-break}
Let $N_\l$ be a one-dimensional $\ElIz$-module with break $y\neq x$.  If $M_\l=M_\l(x)$, then $\max\{x,y\}$ is the unique break of $M_\l\otimes_{\El} N_\l$.
\end{lem}

\begin{proof}
If $r>x$, then \eqref{eqn:break-submodules} implies $I(z)^{(r)}$ acts trivially on $M_\l=M_\l(x)$, so
$$
	(M_\l\otimes_\El N_\l)^{I(z)^{(r)}}
	=
	M_\l\otimes_\El N_\l^{I(z)^{(r)}}
	=
	\begin{cases}
    	M_\l\otimes_\El 0 & x<r\leq y \\
    	M_\l\otimes_\El N_\l & r>x\mbox{ and }r>y.
	\end{cases}
$$
Similarly, if $r>y$, then $I(z)^{(r)}$ acts trivially on and $N_\l=N_\l(y)$
$$
	(M_\l\otimes_\El N_\l)^{I(z)^{(r)}}
	=
	M_\l^{I(z)^{(r)}}\otimes_\El N_\l
	=
	\begin{cases}
		0\otimes_\El N_\l & y<r\leq x \\
		M_\l\otimes_\El N_\l & r>x\mbox{ and }r>y.
	\end{cases}
$$
Therefore,
$$
	(M_\l\otimes_\El N_\l)^{I(z)^{(r)}}
	=
	\begin{cases}
		0 & r\leq x\mbox{ or }r\leq y \\
		M_\l\otimes_\El N_\l & r>x\mbox{ and }r>y,
	\end{cases}
$$
and in particular, $r=\max\{x,y\}$ is the unique break of $M_\l\otimes_{\El}N_\l$ as claimed.
\end{proof}

The \defi{multiplicities} $\m_{1,\l},\ldots,\m_{m,\l}\geq 1$ of $M_\l$ are the positive integers $\m_{i,\l}=\dim_{\El}(M_\l(\b_i))$.

\begin{lemma}\thlabel{lem:non-negative-b_i-for-single-M_l}
If $d=\dim(M_\l)$, then $\b_1,\ldots,\b_m$ are non-negative and lie in $\frac{1}{d!}\bbZ$.
\end{lemma}

\begin{proof}
The product $\b_i\m_{i,\l}$ is a non-negative integer for each $i$ (cf.~\cite[1.9]{Katz:GKM}), and $1\leq\b_i\leq\dim_\El(M_\l)$, so $d!\b_i\in\bbZ$.
\end{proof}

For each $r\geq 0$, we define the \defi{partial Swan conductor} of $M_\l$ to be the finite sum
$$
	\s_{z,\geq r}(M_\l)=\sum_{\b_i\geq r}\,\b_i\m_{i,\l}.
$$
It is the usual Swan conductor $\s_z(M_\l)$ when $r=0$.

\begin{lem}\thlabel{lem:isolating-breaks}
Let $N_\l$ be a one-dimensional $\ElIz$-module with break $y$.  If $1\leq i<m$ and if $\b_i<y<\b_{i+1}$, then
$$
	\s_z(M_\l\otimes N_\l)
	=
	y(\m_{1,\l}+\cdots+\m_{i,\l})+\s_{z,\geq y}(M_\l).
$$
\end{lem}

\begin{proof}
Since $y$ is not a break of $M_\l$, \thref{lem:tensor-break} implies that $\max\{y,\b_j\}$ is the unique break of $M_\l(\b_j)\otimes_\El N_\l$ for $j=1,\ldots,m$.  Therefore the breaks and respective multiplicities of $M_\l\otimes_\El N_\l$ are $y,\b_{i+1},\ldots,\b_m$ and $\m_{1,\l}+\cdots+\m_i,\m_{i+1,\l},\ldots,\m_{m,\l}$, and hence
$$
	\s_z(M_\l\otimes_\El N_\l)
	=
	y(\m_{1,\l}+\cdots+\m_{i,\l})
	+ \b_{i+1}\m_{i+1,\l} + \cdots + \b_m\m_{m,\l}
$$
as claimed.
\end{proof}

\begin{cor}\thlabel{cor:partial-swan-difference}
Let $N_{\l,1},N_{\l,2}$ be one-dimensional $\ElIz$-modules with respective breaks $y_1,y_2$.  If $1\leq i<m$ and if $\b_i<y_1\leq y_2<\b_{i+1}$, then
$$
	\s_z(M_\l\otimes N_{\l,2}) - \s_z(M_\l\otimes N_{\l,1})
	=
	(y_2-y_1)(\m_{1,\l}+\cdots+\m_{i,\l}).
$$
\end{cor}

\begin{proof}
The hypotheses on $y_1,y_2$ imply that
$$
	\s_{z,\geq y_1}(M_\l\otimes_\El N_{\l,1})
	=
	\s_{z,\geq y_2}(M_\l\otimes_\El N_{\l,2}),
$$
and thus the corollary follows from \thref{lem:isolating-breaks}.
\end{proof}

Let $\swp_z(M_\l)$ be the \defi{Swan polygon} of $M_\l$ (see~\cite[1.2]{Katz:GKM} or \cite[pg.~213]{Katz:WR}).  It is the finite polygon in $\bbR^2$ whose vertices are
$$
	\swv_z(\FFl)
	=
	\left\{\,
		(x_j,y_j)
		:
		0\leq j\leq m,\ 
		x_j = \sum_{i\leq j} \m_{i,\l},\ 
		y_j = \sum_{i\leq j} \b_i\m_{i,\l}
	\,\right\}
$$
and whose edges join $(x_j,y_j)$ to $(x_{j+1},y_{j+1})$ for $0\leq j<m$.


\subsection{Compatible systems}

Suppose that $\L$ is finite.  Let $\FFL$ be an $\EL$-compatible system of sheaves on $\C$ of common generic rank $r$.  Let $z\in\C$ be a closed point and $M_\l=\FFlu$ regarded as an $\ElIz$-module, and let
$$
	\s_{z,\geq x}(\FFl):=\s_{z,\geq x}(M_\l)\mbox{ for }x\geq 0,\quad
	\s_z(\FFl):=\s_z(\FFl)
$$
be the partial and usual Swan conductors.  Finally, let
$$
	\swp_z(\FFl):=\swp_z(M_\l),\quad
	\swv_z(\FFl):=\swv_z(M_\l)
$$
be the respective Swan polygon and vertices.

Let $\b_1,\b_2,\ldots,\b_m$ be the increasing sequence of breaks which occur in at least one of $M_\l$, and let $\mu_{i,\l}:=\dim_\El(M_\l(\b_i))$.

\begin{lemma}\thlabel{lem:discrete-b_i}
$\b_1,\b_2,\ldots,\b_m\in\frac{1}{r!}\bbZ_{\geq 0}$.
\end{lemma}

\begin{proof}
Follows from \thref{lem:non-negative-b_i-for-single-M_l}.
\end{proof}

Therefore one can always satisfy the hypotheses of the following lemma.

\begin{lem}\thlabel{lem:swan-independence}
If $s_1,s_2,\ldots,s_{2m}$ is a sequence in $\bbR$ satisfying $\b_i < s_{2i-1} < s_{2i} < \b_{i+1}$ for $1\leq i<m$, then the following are equivalent:

\smallskip
\begin{enum}
\myitem\label{it:si:mu_i} $\mu_{i,\l}$ is \indep{} for $1\leq i\leq m$;
\myitem\label{it:si:partial-swan} $\s_{z,\geq s_j}(\FFl)$ is \indep{} for $1\leq j\leq 2m$;
\myitem\label{it:si:swan-polygon} $\swp_z(\FFl)$ is \indep{};
\myitem\label{it:si:swan-vertices} $\swv_z(\FFl)$ is \indep{}.
\end{enum}
\end{lem}

\begin{proof}
First, \thref{cor:partial-swan-difference} implies that
$$
	\s_{z,\geq s_{2i}}(\FFl) - \s_{z,\geq s_{2i-1}}(\FFl)
	=
	(s_{2i}-s_{2i-1})(\mu_{1,\l}+\cdots+\mu_{i,\l})
	\mbox{ for }
	1\leq i\leq m
$$
so \eqref{it:si:mu_i} and \eqref{it:si:partial-swan} are equivalent.  Second, 
$$
	\swv_z(\FFl)
	=
	\left\{\,
		\left(
			\sum_{1\leq i\leq j}\mu_{i,\l},
			\sum_{1\leq i\leq j}\b_i\mu_{i,\l}
		\right)
		:
		0\leq j\leq m
	\,\right\},
$$
hence \eqref{it:si:mu_i} and \eqref{it:si:swan-polygon} are equivalent.  Finally, the Swan polygon determines and is completely determined by its vertices, hence \eqref{it:si:swan-polygon} and \eqref{it:si:swan-vertices} are equivalent.
\end{proof}


\subsection{Relating Swan polygons and conductors}

Throughout this section we suppose that $\C=\Pone$ and that $\Lambda$ is finite.  Let $n$ be a positive integer which is not divisible by the characteristic of $k$, and let $[n]\colon\Pone\to\Pone$ be the $n$th power map.

\begin{lem}\thlabel{lem:pullback-swan}
$$
	\s_{z,\geq x}([n]^*\FFl)
	=
	\begin{cases}
    	\s_{[n](z),\geq x}(\FFl) & z\neq 0,\infty \\
    	n\cdot\s_{[n](z),\geq x/n}(\FFl) & \mbox{otherwise}
	\end{cases}
$$
\end{lem}

\begin{proof}
On one hand, if $z\neq 0,\infty$, then $[n]$ is unramified over $[n](z)$.  Moreover, the corresponding breaks and multiplicities of $[n]^*\FFl$ at $z$ coincide with those of $\FFl$ at $[n](z)$, and thus
$$
	\s_{z,\geq x}([n]^*\FFl)
	=
	\s_{[n](z),\geq x}(\FFl).
$$
On the other hand, if $z=0$ (resp.~$z=\infty$), then $[n](z)=0$ (resp.~$[n](z)=\infty$) and $[n]$ is totally and tamely ramified over $[n](z)$ since $n$ is coprime to $p$.  Moreover, the breaks of $[n]^*\FFl$ at $z$ are the products of each of the breaks $\b_1,\ldots,\b_m$ of $\FFl$ at $[n](z)$ with $n$ and the corresponding multiplicities $\mu_1,\ldots,\mu_m$ are unchanged (see \cite[pg.~217]{Katz:WR}), and thus
$$
	n\cdot \s_{[n](z),\geq x/n}(\FF_\l)
	=
	n\cdot \sum_{\b_i\geq x/n} \b_i\cdot\m_{i,\l}
	=
	\sum_{n\b_i\geq x} n\b_i\cdot\m_{i,\l}
	=
	\s_{z,\geq x}([n]^*\FF_\l)
$$
as claimed.
\end{proof}

\begin{cor}\thlabel{cor:swan-pullback}
The following are equivalent for all $z\in\Z$:

\smallskip
\begin{enum}
\myitem\label{cor:it:swan-poly-downstairs} $\swp_z(\FFl)$ is \indep;
\myitem\label{cor:it:partial-swan-downstairs} $\s_{z,\geq x}(\FFl)$ is \indep{} for all $x\geq 0$;
\myitem\label{cor:it:partial-swan-upstairs} $\s_{z,\geq x}([n]^*\FFl)$ is \indep{} for all $x\geq 0$;
\myitem\label{cor:it:swan-poly-upstairs} $\swp_z([n]^*\FFl)$ is \indep.
\end{enum}
\end{cor}

\begin{proof}
One one hand, the Swan polygon is completely determined by the partial Swan conductors for all $x\geq 0$, hence \eqref{cor:it:swan-poly-downstairs} and \eqref{cor:it:partial-swan-downstairs} are equivalent and \eqref{cor:it:partial-swan-upstairs} and \eqref{cor:it:swan-poly-upstairs} are equivalent.  On the other hand, \thref{lem:pullback-swan} implies that \eqref{cor:it:swan-poly-downstairs} and \eqref{cor:it:swan-poly-upstairs} are equivalent.
\end{proof}

Recall that each $\FFl$ has generic rank $r$, that $\b_1,\b_2,\ldots,\b_m$ is the increasing sequence of breaks occurring in at least one $M_\l$, and that $\m_{i,\l}=\dim_\El(M_\l(\b_i))$.

\begin{lem}\thlabel{lem:Pone-swan}
Suppose that $z\in\Pone(\Fq)$ and that $E\supseteq\bbQ(\zeta_p)$.  Let $n$ be an integer satisfying $n>3r!$ and $s_0,s_1,s_2,\ldots$ be the sequence of integers given by
$$
	s_j
	=
	\begin{cases}
		0 & \mbox{if }j=0 \\
		\lfloor 1+n\b_i \rfloor & \mbox{if }j=2i-1<2m\mbox{ and }i\in\Zpos \\
		\lfloor 2+n\b_i \rfloor & \mbox{if }j=2i<2m+1\mbox{ and }i\in\Zpos
	\end{cases}.
$$

\smallskip\noindent
Let $z\in\Z$ and $Y=\{z\}$, and, for $0\leq j\leq 2m$, let $\TT_{s_j,\L}$ be the $\EL$-compatible system of sheaves in \thref{lem:artin-schreier-with-integer-swan}.  Then the following are equivalent:
\begin{enum}
\myitem\label{prop:Pone-swan:it:swan-poly-downstairs} $\swp_z(\FFl)$ and $\swv_z(\FFl)$ are \indep{};
\myitem\label{prop:Pone-swan:it:swan-poly-upstairs} $\swp_z([n]^*\FFl)$ and $\swv_z([n]^*\FFl)$ are \indep.
\myitem\label{prop:Pone-swan:it:mu-and-swan} $\mu_{i,\l}$ and $\s_{z,\geq s_j}([n]^*\FFl)$ are \indep{} for $0<i<m+1$ and $0<j<2m+1$;
\myitem\label{prop:Pone-swan:it:swans} $\s_z([n]^*\FFl\otimes_\El\TT_{s_j,\l})$ is \indep{} for $0\leq j<2m+1$.
\end{enum}
\end{lem}

\begin{proof}
\thref{cor:swan-pullback} implies that \eqref{prop:Pone-swan:it:swan-poly-downstairs} and \eqref{prop:Pone-swan:it:swan-poly-upstairs} are equivalent.  The hypotheses that $n>3r!$ implies
$$
	n\b_i < s_{2i-1} < s_{2i} < n\b_{i+1}
	\mbox{ for }
	0<i<m,
$$
so \thref{lem:swan-independence} implies that \eqref{prop:Pone-swan:it:swan-poly-upstairs} and \eqref{prop:Pone-swan:it:mu-and-swan} are equivalent.  Finally, $\TT_{s_0,\l}=\TT_{0,\l}$ is the constant sheaf $\El$ for each $\l\in\L$, and thus \thref{lem:isolating-breaks} implies
$$
	\s_z([n]^*\FFl\otimes_\El\TT_{s_j,\l})
	=
	\begin{cases}
	n\b_1\mu_{1,\l}
	+
	\s_{z,\geq s_1}([n]^*\FFl)
	& \mbox{ if }j=0 \\[0.1in]
	s_j(\mu_{1,\l}+\cdots+\mu_{i,\l})
	+
	\s_{z,\geq s_{j+1}}([n]^*\FFl)
	& \mbox{ if } 0<j<2m \\[0.1in]
	s_j(\mu_{1,\l}+\cdots+\mu_{m-1,\l})
	+
	n\b_m\mu_{m,\l}
	& \mbox{ if }j=2m<\infty.
	\end{cases}
$$

\smallskip\noindent
The left side is \indep{} if and only the right side is, hence \eqref{prop:Pone-swan:it:mu-and-swan} and \eqref{prop:Pone-swan:it:swans} are equivalent.
\end{proof}


\section{Reductions and Constructions}\label{sec:reductions}

In this section we present reductions which allow us to strengthen the hypotheses of \thref{thm:swan:semel}, \thref{thm:swan:bis}, and \thref{thm:swan:ter} respectively, e.g., that $\Fq$ and $\E$ are `sufficiently large' and $\U$ is `sufficiently small.'


\subsection{Field extensions and shrinking $U$}

The following lemma allows us to reduce to the case where $\Fq$ and $\E$ are `sufficiently large' and $\U$ is a `sufficiently small' dense Zariski open subset of $\C$.

\begin{lem}\thlabel{lem:useful-reductions}
Let $\Ep/\E$ be a finite extension and $\L'$ be the primes $\l'$ of $\Ep$ lying over primes in $\L$.  Then any of \thref{thm:swan:semel}, \thref{thm:swan:bis}, and \thref{thm:swan:ter} respectively holds if and only if it holds after any of the following operations:

\smallskip
\begin{enum}
\myitem\label{it:finite-field-extension} replace $\Fq$ by a finite extension $\bbF_{q^n}$;
\myitem\label{it:finite-number-field-extension} replace $\EL$ by $(\Ep,\L')$;
\myitem\label{it:shrink-to-dense-open} replace $\U$ by a dense open subset $\Up$.
\end{enum}
\end{lem}

\begin{proof}
The ramification invariants are geometric so do not change if we replace $\Fq$ by a finite extension $\bbF_{q^n}$, so \eqref{it:finite-field-extension} holds.  Nor do they change if we replace $\El$ by a finite extension $\Ep_{\l'}$, so \eqref{it:finite-number-field-extension} hold.  If $\Up\seq\U$ is a dense Zariski open subset and if $z\in \U\ssm\Up$, then
$$
	\r_z(\FFl)=\r_\El(\FFl),\ \ 
	\d_z(\FFl)=\s_z(\FFl)=\c_z(\FFl)=0
$$
while
$$
	\swv_z(\FFl) = \{\,(0,0),(0,r)\,\},\ \ 
	\ub_z(\FFl) = \{\,(r,1)\,\}
$$
for $r=\rank_{\El}(\FFl)$.  In particular, all of these are \indep, so \eqref{it:shrink-to-dense-open} holds.  The final assertion about being able to restriction to $\L''$ is clear since independence of $\l$ is established individually for each pair $\l,\l'\in\L$.
\end{proof}


\subsection{Reducing to $\C=\Pone$}

The next two lemmas imply that, up to replacing $\Fq$ by a finite extension $\Fqn$ (e.g., so that $\Z\seq\C(\Fqn)$), we may assume $\C=\Pone$ without loss of generality:

\begin{lemma}\thlabel{lem:push-to-P^1}
Suppose that $f\colon\C\to\Pone$ is a finite morphism satisfying $f^{-1}(f(\Z))\seq\C(\Fq)$ and $|f^{-1}(f(\Z))|=\deg(f)|\Z|$, and let $\delta=(\deg(f)-1)\r(\FFl)$.  Then the following hold, for each $z\in Z$:

\smallskip
\begin{enum}
\myitem\label{it:push-to-P^1:EL-compatible-on-V} $\{f_*\FFl\}_{\l\in\L}$ is an $\EL$-compatible system of lisse sheaves on a dense open subset $\V\subseteq f(\U)$;
\myitem\label{it:push-to-P^1:rank-drop} $\r_z(\FFl)=\r_{f(z)}(f_*\FFl)-\delta$ and $\d_z(\FFl)=\d_{f(z)}(f_*\FFl)$;
\myitem\label{it:push-to-P^1:swan-and-total} $\s_z(\FFl)=\s_{f(z)}(f_*\FFl)$ and $\c_z(\FFl)=\c_{f(z)}(f_*\FFl)$;
\myitem\label{it:push-to-P^1:swan-poly} the vertices $\swp_{f(z)}(f_*\FFl)$ and $\swp_z(\FFl)$ satisfy
$$
	\swv_{f(z)}(f_*\FFl)
	=
	\{\,
		(x+\delta,y) : (x,y)\in\swv_z(\FFl)
	\,\}
	\cup
	\{\,(0,0)\,\};
$$
\myitem if $m_1=\delta$ and $m_e=0$ for $e>1$, then
$$
	\ub_{f(z)}(f_*\FFl)
	=
	\left\{\,
		(e,m+m_e)
		:
		(e,m)\in\ub_z(\FFl)
	\,\right\}.
$$
\end{enum}
\end{lemma}

\begin{proof}
Let $\V\subseteq f(\U)$ be a dense Zariski open subset over which $f$ is \etale.  Let $y\in\Pone$ be a closed point and let $x$ vary over $f^{-1}(y)$.  Let $D(x)\subseteq D(y)\subset\piOneV$ be the inclusion of the decomposition groups of $x,y$.  Let $\bar{x}\to x$ be a geometric point and $\II_\ubar$ be the induced $\El[D(y)]$-module $\II_{\bar{u}}=\Ind_{D(x)}^{D(y)}(\FF_{\l,\bar{x}})$, and observe that
$$
	(f_*\FFl)_{\bar{v}}
	=
	\oplus_{f(u)=v}\,\II_\ubar^{\deg(u)/\deg(v)}
$$

\smallskip\noindent
since $f$ is finite (cf.~\cite[II.3.5]{Milne:EC}).

Suppose first that $y=v$ is in $\V$ and thus $x=u$ is in $\U$.  Then $I(u)$ and $I(v)$ are trivial in $\piOneV$ since $f$ is \etale{} over $V$.  In particular, $\tr(\Fr_v\mid\II_{\bar{u}})$ is \indep{} since $\tr(\Fr_u\mid\FF_{\l,\bar{u}})$ is \indep.  Therefore $\tr(\Fr_v\mid(f_*\FFl)_{\bar{v}})$ is \indep{} and \eqref{it:push-to-P^1:EL-compatible-on-V} holds.

Now suppose that $y$ is in $f(\Z)$ and that $x=z$.  The condition that $|f^{-1}(f(\Z))|=\deg(f)|\Z|$ implies that $f$ is also \etale{} over $\Z$ and that the restriction of $f$ to $\Z$ is injective.  In particular, since $\Z\subseteq\C(\Fq)$ by hypothesis, $\FFlz$ is a summand of $(f_*\FFl)_{\bar{v}}$.  Moreover, there are $\deg(f)-1$ other summands $\II_{\bar{u}}$ since $f^{-1}(f(\Z))\seq\C(\Fq)$, and each of them is unramified.  Therefore
$$
	\dim_\El(\II_\ubar)
	=
	\r(\FFl),\ 
	\d_x(\FFl) = \s_x(\FFl) = \c_x(\FFl) = 0
	\mbox{ for }
	x\in f^{-1}(y)\ssm\{z\}
$$
and hence \eqref{it:push-to-P^1:rank-drop} and \eqref{it:push-to-P^1:swan-and-total} hold.  Finally, these summands contribute a horizontal segment from $(0,0)$ to $(\delta,0)$ to $\swp_{f(z)}(f_*\FFl)$ and shift all the vertices of $\swp_z(\FFl)$ to the right by $\delta$, hence \eqref{it:push-to-P^1:swan-poly} holds.
\end{proof}

The next lemma yields a function $f$ which can be used in the previous lemma:

\begin{lem}\thlabel{lem:suitable-f}
Up to replacing $\Fq$ by a finite extension $\Fqn$, there exists a finite morphism $f\colon\C\to\Pone$ such that $|f^{-1}(f(\Z))|=\deg(f)\deg(\Z)$ and $f(\Z)\subseteq\Gm$.
\end{lem}

\begin{proof}
Extend $\Fq$ so that $(\C\ssm\Z)(\Fq)$ contains a point $\infty$.  Let $d\in\Zpos$ satisfy $d\geq 6g+3$ and $p\nmid d$ and let $D$ be the divisor $d\infty$.  Apply \cite[th.~2.2.4]{Katz:TLFM} if $p>2$ or \cite[th.~2.4.4]{Katz:TLFM} if $p=2$ to construct $f\colon\C\to\Pone$ over $\Fqbar$ such that $|f^{-1}(f(\Z))|=\deg(f)\deg(\Z)$.  Replace $f$ with its composition by a general automorphism of $\Pone$ over $\Fqbar$ so that $f(\Z)\subseteq\Gm$, and extend $\Fq$ so that $f$ is defined over $\Fq$.
\end{proof}


\subsection{Artin-Schreier sheaves}

Suppose that $\C=\Pone$.

\begin{lemma}\thlabel{lem:artin-schreier-with-integer-swan}
Suppose that $E\supseteq\bbQ(\zeta_p)$ and that $\Y\sneq\Pone(\Fq)$.  For each integer $s\in\Znn$ coprime to $p$, there exists an $\EL$-compatible system $\TT_{s,\L}$ of rank-one lisse sheaves on $\Pone\ssm\Y$ such that the following hold:

\smallskip
\begin{enum}
\myitem $\s_y(\TT_{s,\l}) = s$ for every $\l\in\L$ and $y\in\Y$;
\myitem if $s=0$, then $\TT_{0,\l}$ is the constant sheaf $\El$ for each $\l\in\L$.
\end{enum}
\end{lemma}

\begin{proof}
If $s=0$, then constant sheaves $\TT_{0,\l}=\El$ clearly have the desired property, so suppose $s$ is positive.

There exists a function $f\colon\Pone\to\Pone$ which has polar divisor $s\Y$.  For example, if $\Y\seq\Aone$ and if $g\in\Fq[x]$ is a square-free polynomial with zero set $\Y$, then one can take $f=g^s$.  Otherwise, if $\infty\in\Y$ and $a\in\Pone(\Fq)\ssm\Y$, then one can construct a function $\Pone\to\Pone$ with polar divisor $(\Y\cup\{a\})\ssm\{\infty\}$ and precompose with any M\"obius transformation which swaps $a$ and $\infty$.  Either way, $f$ is tamely ramified over $\infty$ since $s$ is coprime to $p$.

Let $\psi\colon\bbF_p^\times\to\Et$ be a non-trivial additive character.  For each $\l$, let $\LL_{\psi(x),\l}$ be the Artin--Schreier sheaf corresponding to $\psi$ with coefficients in $\El$.  It is lisse on $\Aone$ and satisfies $\s_\infty(\LL_{\psi(x),\l})=1$, hence the pullback $\TT_{s,\l}=f^*\LL_{\psi(x),\l}$ is lisse on $\Pone\ssm\Y$ and satisfies $\s_y(\TT_{s,\l})=s$ for every $y\in\Y$ since $f$ is tamely ramified over $\infty$.  Moreover, the system $\{\LL_{\psi(x),\l}\}_{\l\in\L}$ is $\EL$-compatible by construction, hence so is the pullback system $\TT_{s,\L}=\{\TT_{s,\l}\}_{\l\in\L}$.  Compare \cite[pg.~217]{Katz:WR}.
\end{proof}


\section{Proof of \thref{thm:swan:semel}}\label{sec:proof-of-swan-thms}

In this section we proof the implications
$$
	\mbox{\thref{thm:swan:ter}}
	\ \Rightarrow
	\mbox{\thref{thm:swan:bis}}
	\ \Rightarrow
	\mbox{\thref{thm:swan:semel}}
$$
and then we prove \thref{thm:swan:ter}.


\subsection{\thref{thm:swan:ter} implies \thref{thm:swan:bis}}\label{subsec:ter-implies-first}

Suppose that $\C=\Pone$ and that $\L$ is finite.  By \thref{lem:useful-reductions}, we may replace $\Fq$ and $\E$ by finite extensions and suppose without loss of generality that $\Z\seq\C(\Fq)$ and that $\E\supseteq\bbQ(\zeta_p)$.  Then Theorem~\ref{thm:swan:ter} implies that the equivalent conditions of \thref{lem:Pone-swan} hold, for each $z\in\Z$, and hence it implies Theorem~\ref{thm:swan:bis} as claimed.


\subsection{\thref{thm:swan:bis} implies \thref{thm:swan:semel}}\label{subsec:bis-implies-semel}

By \thref{lem:useful-reductions}, we may replace $\Fq$ by a finite extension and suppose without loss of generality that $\Z\seq\C(\Fq)$ and that there is a morphism $f\colon\C\to\Pone$ satisfying the hypotheses of \thref{lem:push-to-P^1}.  Therefore, up to replacing $\FFL$ by $f_*\FFL$ and $U$ by a dense Zariski open $V\seq f(U)$, we may suppose without loss of generality that $\C=\Pone$.  Since it suffices to prove Theorem~\ref{thm:swan:semel} for each finite subset $\L'\seq\L$,  Theorem~\ref{thm:swan:bis} implies it as claimed.


\subsection{Proof \thref{thm:swan:ter}}

Suppose that $\C=\Pone$ and that $\L$ is finite.  Once again, by \thref{lem:useful-reductions}, we may replace $\Fq$ and $\E$ by finite extensions and suppose without loss of generality that $\Z\seq\C(\Fq)$ and that $\E\supseteq\bbQ(\zeta_p)$.

Let $z\in\Z$ and $\Y=\Z\ssm\{z\}$.  The Euler-Poincare formula for the Euler characteristic of a lisse $\El$-sheaf $\GGl$ on $U$ is given by
\begin{equation}\label{eqn:euler-poincare}
	\chi(\U,\GGl)
	=
	\r_\El(\GGl)\cdot\chi(\U,\El)
	- \s_z(\GGl)
    - \smallsum_{y\in \Y}\,\s_y(\GGl)
\end{equation}
since $\Y\sub\Z\seq\C(\Fq)$ (cf.~\cite[2.3.1]{Katz:GKM}).  Moreover, the Euler characteristic $\chi(\U,\GGl)$ is \indep{} if $\GGl$ is part of a compatible system $\GGL$ since it it is the negative of the degree of
$$
	L(T,\GGl) = \prod_{u\in U}\det\left(1 - T\,\Fr_u\mid\GGlu\right)^{-1}
$$
and all terms on the right are \indep.

Let $s$ be any positive integer which is coprime to $p$ and which exceeds $\s_y(\FFl)$ for every $y\in\Y$ and $\l\in\L$, and let $\TT_{s,\L}$ be the $\EL$-compatible system of \thref{lem:artin-schreier-with-integer-swan}.  Then \thref{lem:tensor-break} implies that
$$
	\s_z(\FFl\otimes\TT_{s,\l}) = \s_z(\FFl)
	,\ \ 
	\s_y(\FFl\otimes\TT_{s,\l}) = r\cdot \s_y(\TT_{s,\l})\mbox{ for }y\in\Y
$$
where $r=\rank_\El(\FFl)$.  Applying \eqref{eqn:euler-poincare}  with $\GGl=\FFl\otimes\TT_{s,\l}$ and rearranging terms yields
$$
	\s_z(\FFl)
	=
	r\cdot\chi(U,\El)
	-
	\chi(U,\FFl\otimes\TT_{s,\l})
	-
	r\cdot s\cdot|Y|,
$$
and in particular, everything on the right is \indep, so $\s_z(\FFl)$ is also \indep{} as claimed.


\section{Proof of \thref{thm:pure}}\label{sec:proof-of-purity-theorem}

Let $\FFL$ be an $\EL$-compatible system of lisse sheaves on $U$ of common rank $r$, and suppose it is pointwise pure of weight $w$.  Let
$$
	L(T,j_*\FFl)
	=
	\prod_{c\in\C}
	L(T^{\deg(c)},\FF_{\l,c})^{-1}	
$$
where $c$ varies over the closed points of $\C$.

\begin{thm}\thlabel{thm:independence-of-missing-eulers}
Suppose the hypotheses of \thref{thm:pure} hold.  Then, for each $z\in\Z$, the Euler factor $L(T,\FF_{\l,z})$ lies in $\E[T]$ and is \indep, and thus so are $\r_z(\FFl)$, $\d_z(\FFl)$, and $\c_z(\FFl)$. 
\end{thm}

\begin{proof}
Using Deligne's theorem, Katz showed that $L(T,\FF_{\l,z})$ lies in $\E[T]$ and is \indep{} (see \cite[Appendix]{Katz:WR}).  Since
$$
	\r_z(\FFl) = \deg(L(T,\FF_{\l,z}) = r - \d_z(\FFl),
$$
it follows immediately that $\r_z(\FFl)$ and $\d_z(\FFl)$ are \indep.  Moreover,
$$
	\c_z(\FFl) = \r_z(\FFl) + \s_z(\FFl)
$$
is \indep{} by Theorem~\ref{thm:swan:semel}.
\end{proof}

It remains to show that $\ub_z(\FFl)$ is \indep.  By \thref{lem:useful-reductions}, we may suppose without loss of generality that $\Z\seq\C(\Fq)$.  Let $z\in\Z$, and consider a factorization
$$
	L(T,\FF_{\l,z})
	=
	\prod_{i=1}^r
	(T-\alpha_{z,i})
$$
over $\Ebar$.  For each field embedding $\iota\colon\Ebar\to\bbC$  and integer $s$, let
$$
	\wm_{\iota,s,z}(\FFl)
	=
	\left|\{\,
		i
		:
		|\iota(\alpha_i)|^2 = (1/q)^s
	\,\}\right|,
$$
and let
$$
	\w_{\iota,z}(\FFl)
	=
	\{\,
		(e,m)
		:
		m=\wm_{w+e-1,z}(\FFl)
		\mbox{ and }
		m>0
	\,\}.
$$
If $\d_z(\FFl)=0$, that is, if $\FFl$ is lisse over $U\cup\{z\}$, then one can show that
$$
	\wm_{\iota,s,z}(\FFl)
	=
	\begin{cases}
	r & s=w \\
	0 & s\neq w
	\end{cases}
$$
since $\deg(z)=1$.  The following proposition deals with the general case:

\begin{prop}
$\w_{\iota,z}(\FFl)=\ub_z(\FFl)$.
\end{prop}

\begin{proof}
Follows from \cite[1.16.2--3 and 1.8.4]{Deligne:WeilII} (cf.~\cite[7.0.7]{Katz:GKM}).
\end{proof}

\noindent
In particular, not only is $\w_{\iota,z}(\FFl)$ independent of $\iota$, it is \indep{} by \thref{thm:independence-of-missing-eulers}.  Therefore $\ub_z(\FFl)$ is \indep{} as claimed.


\section{$\EGL$-compatible Sytems}\label{sec:EGL-compatible}

Let $G$ be a finite group, $\E/\bbQ$ be a number field, and $\L$ be a finite set of non-zero primes $\l\sub\Z_\E$ not dividing $q$.  Let $\C/\Fq$ be a proper smooth geometrically connected curve and $\U\seq\C$ be a dense Zariski open subset.  For each $\l\in\L$, let $\FF_\l$ be a lisse sheaf on $U$ of $\El[G]$-modules and $V_\l$ be the geometric generic fiber $\FFlu$.

Given a dense Zariski open subset $\X\seq\C$ defined over $\Fq$, we say that the system $\FFL=\{\FFl\}_{\l\in\L}$ is \defi{$\EGL$-compatible on $\X$} (resp.~\defi{weakly $\EGL$-compatible}) iff for every closed point $x\in\X$, every integer $m\geq 0$ (resp.~$m=0$), and every element $g\in G$, the trace
$$
	\tr\left(g\cdot\Fr_x^m\act\VlIx\right)
$$
lies in $\E$ and is \indep{}.

We say that $\FFL$ is \defi{(pointwise) pure of weight $w$} on $U$ iff every $\FFl$ is pointwise pure of weight $w$ as a lisse $\El$-sheaf on $U$.

\begin{thm}\thlabel{thm:EGL-compatible}
Let $\Fq$ be a finite field, $\C/\Fq$ be a proper smooth geometrically connected curve, and $j\colon\U\to\C$ be the inclusion of a dense Zariski open subset over $\Fq$.  Let $\E/\bbQ$ be a number field, $G$ be a finite group, $\L$ be a set of non-zero primes $\l\sub\bbZ_\E$ not dividing $q$, and $\FFL=\{\FFl\}_{\l\in\L}$ be a system of lisse sheaves on $\U$.  Then the following hold:

\smallskip
\begin{enum}
\myitem\label{it:thm:weakly-EGL} If $\FFL$ is weakly $\EGL$-compatible on $\U$, then $j_*\FFL$ is weakly $\EGL$-compatible on $\C$.
\myitem\label{it:thm:direct-image-is-EGL} If $\FFL$ is $\EGL$-compatible and pure of weight $w$ on $\U$, then $j_*\FFL$ is $\EGL$-compatible on $\C$.
\end{enum}
\end{thm}

\noindent
If $G$ acts trivially on each $\Vl$ (e.g., if is the trivial group), then $\FFL$ is $\EGL$-compatible if and only if it is $\EL$-compatible, in which case the theorem follows from \thref{thm:independence-of-missing-eulers}.  The proof of \thref{thm:EGL-compatible}, which uses \thref{thm:pure}, will occupy the remainder of the section.

Let $\FFlG\subseteq\FFl$ be the $\El[G]$-subsheaf of $G$-invariants; it is the lisse $\El$-sheaf on $U$ whose geometric generic fiber is $\VlG$.

\begin{lem}\thlabel{lem:EGL-invariants}
If $\FF_\L$ is $\EGL$-compatible on $\U$, then so is $\{\,\FFlG\,\}_{\l\in\L}$.
\end{lem}

\begin{proof}
Let $\pi\in\End_{\El}(\FFlu)$ be the idempotent $\frac{1}{|G|}\sum_{h\in G}h$.  It is projection onto $\VlG$ and
$$
	\tr(g\cdot\Fr_u^m\mid\VlG)
	= \tr(g\cdot\Fr_u^m\cdot\pi\mid\Vl)
	= \smallfrac{1}{|G|}\sum_{h\in G}\,\tr(gh\cdot\Fr_u^m\act\FFlu)
$$
for each integer $m\geq 0$ and element $g\in G$.  In particular, the last term of the display is \indep{} if $\FFL$ is $\EGL$-compatible on $\U$, thus so is the first.
\end{proof}

Let $M$ be a finite-dimensional $\E[G]$-module and $\Ml$ be the constant sheaf $M\otimes_\E \El$ on $\U$.

\begin{lem}\thlabel{lem:EGL-tensor-with-constant}
If $\FFL$ is $\EGL$-compatible on $\U$, then so is $\{\,\Ml\otimes_{\El}\FFl\,\}_{\l\in\L}$.
\end{lem}

\begin{proof}
The right side of the identity
$$
	\tr(g\cdot\Fr_u^m\act \Ml\otimes_{\El}\FFl)
	= \tr(g\act \Ml)\cdot\tr(g\cdot\Fr_u^m\act \FFl)
$$
is \indep{} if $\FFL$ is $\EGL$-compatible, thus so is the left.
\end{proof}

Let $\DM$ be the $\E$-dual of $M$ as $\E[G]$-module, $\DMl$ be the constant sheaf $\DM\otimes_\E\El$ on $\U$, and $\HH(\Ml,\FFl)=(\DMl\otimes_{\El}\FFl)^G$.

\begin{lem}\thlabel{lem:HMF}
Suppose $\FFL$ is $\EGL$-compatible on $\U$.  Then the following hold:

\smallskip
\begin{enum}
\myitem\label{it:HMF-is-EGL-compatible} $\{\,\HH(\Ml,\FFl)\,\}_{\l\in\L}$ is $\EGL$-compatible on $\U$;
\myitem\label{it:HMF-is-pure} if $\FFl$ is pure of weight $w$, then so is $\HH(\Ml,\FFl)$.
\end{enum}
\end{lem}

\begin{proof}
\thref{lem:EGL-invariants} and \thref{lem:EGL-tensor-with-constant} imply \eqref{it:HMF-is-EGL-compatible}.  The sheaf $\DMl$ is pure of weight 0.  Therefore the sheaf $\DMl\otimes_\El\FFl$ and the subsheaf $\HH(\Ml,\FFl)$ are pure of weight $w$, so \eqref{it:HMF-is-pure} holds.
\end{proof}

Extend $\E$ so that every simple $\EG$-module is absolutely simple.  Let $\Z=\C\ssm\U$, and for each $z\in\Z$, let $\zbar\to z$ be a geometric point.

\begin{lem}\thlabel{lem:M-multiplicity-in-HMF}
If $M$ is simple, then its multiplicity in $\VlIz$ equals $\r_z(\HH(\Ml,\FFl))$ for each $z\in\Z$.
\end{lem}

\begin{proof}
We have the identities
$$
	(j_*\HH(\Ml,\FFl))_\zbar
	= ((j_*(\DMl\otimes_{\El}\FFl))^G)_{\etabar}^{I(z)}
	= ((j_*(\DMl\otimes_{\El}\FFl))_{\etabar}^{I(z)})^G
	= (\DM\otimes_\E\El\otimes_{\El}j_*(\FFl)_{\zbar})^G
$$
since the actions of $G$ and $I(z)$ commute and $I(z)$ acts trivially on $\DMl$.  The last term equals $\Hom_{\ElG}(M\otimes_\E\El,j_*(\FFl)_\zbar)$, and its $\El$-dimension is the desired multiplicity since $M$ is absolutely simple.
\end{proof}

Let $M_1,M_2,\ldots$\ be the (isomorphism classes of) simple $\EG$-modules and $\t_i\colon G\to \E$ be the character of $M_i$.

\begin{lem}\thlabel{lem:independence-of-multiplicities}\ 
The following hold, for each $z\in\Z$:

\smallskip
\begin{enum}
\myitem\label{it:lem:independence-of-multiplicities:1} if $\FFL$ is pure, then the multiplicity $m_i$ of $M_{i,\l}=M_i\otimes_\E\El$ in $j_*(\FFl)_\zbar$ is \indep;
\myitem\label{it:lem:independence-of-multiplicities:2} if $g\in G$, then $\tr(g\mid j_*(\FFl)_{\zbar}) = \sum_i m_i\cdot\t_i(g)$ and thus is \indep.
\end{enum}
\end{lem}

\begin{proof}
\thref{lem:M-multiplicity-in-HMF} and \thref{thm:pure} imply $m_i=\r_z(\HH(M_{i,\l},\FFl))$ is \indep{} if $\FFL$ is pure, so (\ref{it:lem:independence-of-multiplicities:1}) holds.  Moreover, $j_*(\FFl)_{\zbar}=\oplus_iM_{i,\l}^{\oplus m_i}$ by definition, so (\ref{it:lem:independence-of-multiplicities:2}) holds.
\end{proof}

\noindent
In particular, \thref{lem:independence-of-multiplicities}.\ref{it:lem:independence-of-multiplicities:2} implies \thref{thm:EGL-compatible}.\ref{it:thm:weakly-EGL}.

Let $K\subseteq G$ be a conjugacy class and $\delta\colon G\to\{0,1\}$ be its characteristic function.

\begin{lem}
There exist $a_1,a_2,\ldots\in \E$ satisfying $\delta=\sum_i a_i\t_i$.
\end{lem}

\begin{proof}
The $\t_i$ form an $\E$-basis of the space of characters $G\to \E$, and $\delta$ lies in that space.
\end{proof}

\noindent
Therefore, if $k\in K$ and $z\in\Z$, then
\begin{eqnarray*}
|K|\cdot \tr(k^{-1}\cdot\Fr_z^m\mid j_*(\FFl)_{\zbar})
	& = & \smallsum_g\,\delta(g^{-1})\cdot\tr(g\cdot\Fr_z^m\mid j_*(\FFl)_{\zbar}) \\
	& = & \smallsum_{i,g}\,a_i\cdot\t_i(g^{-1})\cdot\tr(g\cdot\Fr_z^m\mid j_*(\FFl)_{\zbar}) \\
	& = & |G|\cdot\smallsum_i\,a_i\cdot\tr(\Fr_z^m\mid\HH(M_{i,\l},\FFl)_\zbar)
\end{eqnarray*}
Compare \cite[pg.~171]{Katz:CC} for the last identity.  In particular, \thref{lem:HMF}.2 and \thref{thm:pure} imply the last expression is \indep, hence \thref{thm:EGL-compatible}.\ref{it:thm:direct-image-is-EGL} holds.


\bibliography{ramification-type--arxiv-v3}{}

\begin{thebibliography}{1}

\bibitem{Deligne:Constantes}
P.~Deligne.
\newblock Les constantes des \'equations fonctionnelles des fonctions {$L$}.
\newblock In {\em Modular functions of one variable, {II} ({P}roc. {I}nternat.
  {S}ummer {S}chool, {U}niv. {A}ntwerp, {A}ntwerp, 1972)}, pages 501--597.
  Lecture Notes in Math., Vol. 349. Springer, Berlin, 1973.

\bibitem{Deligne:WeilII}
Pierre Deligne.
\newblock La conjecture de {W}eil. {II}.
\newblock {\em Inst. Hautes \'Etudes Sci. Publ. Math.}, (52):137--252, 1980.

\bibitem{Katz:CC}
Nicholas~M. Katz.
\newblock Crystalline cohomology, {D}ieudonn\'e modules, and {J}acobi sums.
\newblock In {\em Automorphic forms, representation theory and arithmetic
  ({B}ombay, 1979)}, volume~10 of {\em Tata Inst. Fund. Res. Studies in Math.},
  pages 165--246. Tata Inst. Fundamental Res., Bombay, 1981.

\bibitem{Katz:WR}
Nicholas~M. Katz.
\newblock Wild ramification and some problems of ``independence of {$l$}''.
\newblock {\em Amer. J. Math.}, 105(1):201--227, 1983.

\bibitem{Katz:GKM}
Nicholas~M. Katz.
\newblock {\em Gauss sums, {K}loosterman sums, and monodromy groups}, volume
  116 of {\em Annals of Mathematics Studies}.
\newblock Princeton University Press, Princeton, NJ, 1988.

\bibitem{Katz:TLFM}
Nicholas~M. Katz.
\newblock {\em Twisted {$L$}-functions and monodromy}, volume 150 of {\em
  Annals of Mathematics Studies}.
\newblock Princeton University Press, Princeton, NJ, 2002.

\bibitem{Milne:EC}
James~S. Milne.
\newblock {\em \'{E}tale cohomology}, volume~33 of {\em Princeton Mathematical
  Series}.
\newblock Princeton University Press, Princeton, N.J., 1980.

\end{thebibliography}
\bibliographystyle{plain}


\end{document}